\documentclass[12pt]{amsart}

\def\diam{{\rm diam}}

\def\eps{{\varepsilon}}
\def\teps{{\tilde\varepsilon}}

\def\mes{{\rm mes}}

\def\Card{{\rm Card}}

\def\Prob{{\mathbb{P}}}

\def\Vol{{\rm Vol}}
\def\EXP{{\mathbb{E}}}

\def\naturals{\mathbb{N}}

\def\reals{\mathbb{R}}

\def\integers{\mathbb{Z}}

\def\RmII{{I\!\!I}}

\def\bk{\mathbf{k}}

\def\brC{{\bar C}}

\def\brH{{\bar H}}

\def\brN{{\bar N}}

\def\brv{{\bar v}}

\def\brx{{\bar x}}

\def\cC{\mathcal{C}}

\def\cI{\mathcal{I}}

\def\cG{\mathcal{G}}

\def\cL{\mathcal{L}}

\def\cN{\mathcal{N}}

\def\cR{\mathcal{R}}

\def\cS{\mathcal{S}}

\def\fA{\mathfrak{A}}

\def\fM{\mathfrak{M}}

\def\fg{\mathfrak{g}}

\def\fl{\mathfrak{l}}

\def\fm{\mathfrak{m}}

\def\hX{{\hat X}}

\def\heps{{\hat{\eps}}}

\def\tH{{\tilde H}}

\def\tT{{\tilde T}}

\def\teps{{\tilde\varepsilon}}

\makeatother

\usepackage{xcolor}

\usepackage[margin=3cm]{geometry}
\usepackage{amsthm, bbm}
\usepackage{amsmath}
\usepackage{bbm, mathrsfs, psfrag}
\usepackage{hyperref}

\numberwithin{equation}{section} 

\theoremstyle{definition}
\newtheorem{definition}{Definition}[section]
\newtheorem{theorem}[definition]{Theorem}
\newtheorem{lemma}[definition]{Lemma}
\newtheorem{corollary}[definition]{Corollary}
\newtheorem{proposition}[definition]{Proposition}
\newtheorem{remark}[definition]{Remark}

\definecolor{OliveGreen}{rgb}{0,0.6,0}

\def\one{{\mathbbm{1}}}
\def\DS{\displaystyle}
\def\d{\mathrm d}

\makeatletter
\@namedef{subjclassname@1991}{2020 Mathematics Subject Classification}
\makeatother


\usepackage[normalem]{ulem}

\author[D. Dolgopyat, C. Dong, A. Kanigowski, P. N\'andori.]
{D. Dolgopyat$^1$, C. Dong$^2$, A. Kanigowski$^1$$^{,4}$, P. N\'andori$^3$.}

\date{\today\\
    $^1$University of Maryland, College Park, MD, USA\\%
    $^2$Chern Institute of Mathematics and LPMC, Nankai University, Tianjin, China
\\%
    $^3$Yeshiva University, New York, NY, USA\\%
 $^4$Jagiellonian University, Krakow, Poland\\%
 }

\title{Limit Theorems for low dimensional generalized $T,T^{-1}$ transformations}


\begin{document}

\maketitle

\begin{abstract}
 We consider generalized $(T, T^{-1})$ transformations such that the base map satisfies a
multiple mixing local limit theorem and  anticoncentration large deviation bounds and in the fiber we have 
$\reals^d$ actions with $d=1$ or $2$ which are exponentially mixing of all orders. If the skewing cocycle has
zero
drift, we show that the ergodic sums satisfy the same limit theorems as the random
walks in random scenery studied by Kesten and Spitzer (1979) and Bolthausen (1989). The proofs rely on the quenched 
CLT for the fiber action
and the control of the quenched variance. This paper complements our previous work where
the classical central limit theorem is obtained for a large class of generalized $(T, T^{-1})$ transformations.
\end{abstract}


\section{Introduction}

This work is a continuation of our study on generalized $T,T^{-1}$ transformations, following previous work \cite{DDKN,DDKN2,survey}. Our main innovation in this paper is to provide several limit theorems in the low dimensional setting, complementing to the higher dimensional case in \cite{DDKN2}.

\subsection{Results}

Let $f$ be a smooth map of a manifold $X$ preserving a measure $\mu,$
and $G_t$ be an $\reals^d$  action on a manifold $Y$
preserving a measure $\nu.$

\begin{definition}
$G_t$ enjoys multiple exponential mixing of all orders if
 there is $\alpha>0$ such that for each $r$ 
there are constants $C, c>0$ such that
for all zero mean $C^\alpha$ functions $A_1, \dots, A_r$, for all $t_1,\dots t_r\in \reals^d$
\begin{equation}
\label{MEM}
\left|\nu\left(\prod_{j=1}^r A(G_{t_j} y) \right)\right|\leq C \left[\prod_{j=1}^r \|A\|_{C^\alpha} \right] e^{-c\fl} 
\end{equation}
where $\fl$ is the gap $\DS \fl=\max_{i\neq j} \|t_i-t_j \|.$
\end{definition}

\begin{remark}
A simple interpolation shows that if \eqref{MEM} holds for some $\alpha$ then it also holds for all $\alpha$
(see e.g. \cite[Appendix A]{DFL}).
\end{remark}

In case $d = 1$, there are plenty of
examples of multiply exponentially mixing systems see e.g. the discussion in \cite{DDKN}.
In the case $d = 2$ our main example is the following:
 $Y=SL_{3}(\reals)/\Gamma$, {where $\Gamma$ is a cocompact lattice,}
 $G_t: Y\to Y$ is the  Cartan action 
on $Y$, and $\nu$ is the Haar measure. 
 More generally one can consider subactions of Cartan actions on $\cG/\Gamma$ 
where $\cG$ is a semisimple Lie group with compact factors and $\Gamma$ is a cocompact lattice.
In particular,
we can take $Y=SL_{d}(\reals)/\Gamma$, and let $G_t$ be an action by 
a two dimensional subgroup of the  group of diagonal matrices.

Let $\tau: X\to\reals^d$ be a smooth map.
We study the  map $F: (X\times Y)\to (X\times Y)$ given by
\begin{equation}
\label{DefTTInv}
F(x,y)=(f(x), G_{\tau(x)}y). 
\end{equation}
Note that $F$ preserves the measure $\zeta=\mu\times \nu$ and that 
$$ F^N(x,y)=(f^N x, G_{\tau_N(x)} y)\quad\text{where}\quad 
\tau_N(x)=\sum_{n=0}^{N-1} \tau(f^n x). $$
Let $H: X\times Y\to \reals$ be a sufficiently smooth function with $\zeta(H) = 0$ and
$$ S_N=\sum_{n=0}^{N-1} H(F^n(x,y)). $$

We want to study the distribution of $S_N$ when the initial condition $(x,y)$ is distributed according to
$\zeta.$

We shall assume that $f$ is ergodic and satisfies the CLT for H\"older functions.

\begin{definition}
\label{DfnMixLLT}
We say that $\tau$ satisfies the {\em mixing local limit theorem} 
(MLLT) if 
for any sequences $(\delta_n)_{n \in \mathbb N} \in \reals$, with
$\displaystyle \lim_{n\to\infty} \delta_n = 0$ and
$(z_n)_{n\in \naturals}\in \reals^d$ such that 
$ |\frac{z_n}{\sqrt{n}}- z| < \delta_n$ for any cube $\cC\subset \reals^d$ and any continuous functions $A_0, A_1:X\to\reals$
$$ \lim_{n\to\infty} n^{d/2} \mu\Big(A_0(\cdot) A_1(f^n\cdot) \one_\cC(\tau_n-z_n)\Big)=
\fg(z) \mu(A_0) \mu(A_1) \Vol(\cC)$$
where $\fg(z)$ is a  non-degenerate Gaussian density
and the convergence is uniform  once 
$(\delta_n)_{n \in \mathbb N}$ is fixed and 
$A_0, A_1, z$ range over compact subsets of $C(X), C(X)$ and $\reals^d$.
\end{definition}

\begin{definition}
$\tau$ satisfies {\em multiple mixing local limit theorem}
(MMLLT) if for each $m\in \naturals$
for any sequence $(\delta_n)_{n \in \mathbb N} \in \reals$, with
$\displaystyle \lim_{n\to\infty} \delta_n = 0$,  for any
family of sequences $(z_n^{(1)}, \dots, z_n^{(m)})_{n\in \naturals}$ with
$ |\frac{z_n^{(j)}}{\sqrt{n}} - z^{(j)}| < \delta_n$
for any cubes $\{\cC_j\}_{j\leq m}\subset \reals^d$ and any continuous functions $A_0, A_1, \dots A_m:X\to\reals$
for any sequences $n_1 \dots n_m\in \naturals$ such that $n_{j+1}-n_j\to\infty$ (with $n_0=0$),
$$ \lim_{\min|n_j-n_{j'}|\to\infty} \left(\prod_{j=1}^m \left(n_j-n_{j-1}\right)^{d/2}\right) 
\mu\left(\prod_{j=0}^m A_j(f^{n_j} \cdot) 
\prod_{j=1}^m \one_{\cC_j}\left(\tau_{n_j}
-z_{n_j}^{(j)}\right)\right)$$
$$=
\prod_{j=0}^m \mu(A_j) \prod_{j=1}^m \fg\left(z^{(j)}-z^{(j-1)}\right) \prod_{j=1}^m \Vol(\cC_j),
$$
where $z^{(0)}=0.$ Moreover, the convergence is uniform  once 
$(\delta_n)_{n \in \mathbb N}$ is fixed, 
 $A_0, A_1,\dots, A_m$ range over compact subsets of $C(X)$ and $z^{(j)}$ range over 
a compact subset of $\reals^d$ for every $j \leq  m$.
\end{definition}

\begin{definition}
  $\tau$ satisfies the {\em anticoncentration large deviation bound of order $s$} 
  if
  there exist a constant $K$ and a decreasing function $\Theta$ such that
  $\DS \int_1^\infty \Theta(r) r^{d}\d r<\infty$ and for any numbers $n_1, n_2, \dots, n_s$,
  for any unit cubes $C_1, C_2, \dots, C_s$ centered at $c_1, c_2, \dots, c_s$ 
\begin{equation}
\label{eq:ACLD}
 \mu\left(x: \tau_{n_j}\in C_j \text{ for } j=1,\dots, s\right)
\leq 
\end{equation}
$$K \left(\prod_{j=1}^s \left(n_{j}-n_{j-1}\right)^{-d/2}\right) 
\Theta\left(\max_{j} \frac{\|c_j-c_{j-1}\|}{ \sqrt{n_j-n_{j-1}}}\right) . $$

\end{definition}  

The class of maps satisfying the MMLLT and the large deviation anticoncentration 
bounds includes in particular the maps 
which admit Young towers with exponential tails, see \cite{SV04, Pene09}.


\begin{theorem}
\label{ThCLT2}
Suppose that $d=2$, $G_t$ enjoys multiple exponential mixing of all orders,
$f$ satisfies the CLT for smooth observables and
$\tau$ satisfies the MMLLT and
the anticoncentration large deviation bound.
Then there is $\Sigma^2$ such that
$\DS \frac{S_N}{\sqrt{N \ln N}}$ converges as $N\to \infty$ to the normal distribution with
zero mean and variance $\Sigma^2.$ 
\end{theorem}


\begin{theorem}
\label{ThCLT1}
Suppose that $d=1$, $G_t$ enjoys multiple exponential mixing of all orders,
$f$ satisfies the CLT for smooth observables and
$\tau$ satisfies the MMLLT and
the anticoncentration large deviation bound.
Then
there is a constant $\Sigma$ such that
\begin{equation}
\label{KeSp}
 \frac{S_N}{N^{3/4}} \quad \text{converges as $N\to \infty$ to a product} \quad 
\Sigma \cL \cN,
\end{equation}
 where $\cL$ and $\cN$ are independent,
$\cN$ has standard normal distribution and
\begin{equation}
\label{DefL}
 \cL=\sqrt{\int_{-\infty}^\infty \ell_x^2 \d x\;\;} 
\end{equation}
where $\ell_x$ is the local time of the standard Brownian Motion at time $1$ and spatial location
$x.$ 
\end{theorem}

\begin{remark}
If $d=1$ then the exponential mixing condition can be weakened to sufficiently fast polynomial mixing,
see Theorem \ref{ThCLT1Poly} in Section \ref{ScPoly}. It seems more difficult to weaken the mixing assumption
in two dimension. We do not pursue this topic since we do not know examples
of smooth
$\reals^2$ actions
which are mixing at a fast polynomial (rather than exponential) speed.
\end{remark}

\begin{remark}
 The asymptotic variance $\Sigma^2$ in Theorems \ref{ThCLT2} and 
\ref{ThCLT1} has similar form (see \eqref{DefSigma2} and \eqref{defLambda}). Namely let 
$\DS \tH(y)=\int_X H(x,y) d\mu(x).$ Then
$$ \Sigma^2=c_d \int_{\reals^d} \int_Y \tH(y) \tH(G_t y) d\nu(y) 
dt. $$
In dimension $1$ if $\Sigma^2=0$ then there is an $L^2$ function $J(y)$ such that 
$$ \int_0^t \tH(G_s y) ds=J(G_t y)-J(y).$$
This was shown in the discrete case by Rudolph, \cite{Rud} (Proposition 2 and Lemma 7), but the proof applies with no changes in the continuous case.

It is shown in \cite[Theorem 6.4]{DDKN2} that for $\reals^1$ volume preserving 
exponentially mixing flows,
the quadratic
form $H\mapsto \Sigma^2(H)$ is not identically zero. Therefore the set of functions of vanishing asymptotic 
variance is a proper linear subspace.
The proof of Theorem 6.4 in \cite{DDKN2} works also for higher dimensional (not necessarily volume
preserving actions)
provided that 

\noindent
(i) there are constants $K_1, \beta_1, K_2, \beta_2$ such that for all $y$ and $r$,
$\DS K_1 r^{\beta_1}\!\!\leq\!\! \nu(B(y, r))\!\! \leq\!\! K_2 r^{\beta_2} $ 
and

\noindent
(ii) there exists
a {\em slowly recurrent point}, that is, a point $y \in Y$ such that for all positive constants $K$ and $A$ 
there exists $r_0(K,A)$ so that for all $r < r_0$ and for all $t$ with $|t| < K$, $\DS \nu(B(y, r) G_{-t} B(y,r))\leq \frac{\nu(B(y,r))}{|\ln r|^A}$,
where $B(y,r)$ is the ball of radius $r$ centered at $y$.
(We note that for 
exponentially mixing volume preserving $\reals$ 
actions, almost all points are slowly recurrent, see \cite{DFL}.) 
 \smallskip

We believe that in many cases the vanishing of the asymptotic variance  entails that the ergodic sums of $H$ 
satisfy the classical CLT but this will be a subject of a future work.
\end{remark}

\begin{remark}
Results analogous to the above theorems
can be proved in case $G$ is an action of $\integers^d$
($d = 1,2$)
and $\tau:X\to\integers^d$ is a piecewise smooth map satisfying the appropriate assumptions such as
the CLT,
MMLLT, and anticoncentration large deviations bounds. Since the
results as well as the proofs are virtually the same, we omit the case of 
discrete actions.  One can also take $X$ to be a subshift of finite type, see the discussion below.
\end{remark}

\subsection{Discussion}
\label{subsec:RWRS}

Here, we discuss previous results related to Theorems \ref{ThCLT2}, \ref{ThCLT1}.

The first results about $T, T^{-1}$ transformations
 pertain to so called {\em random walks in random scenery}. In this model
we are given a sequence $\{\xi_z\}_{z\in \integers^d}$ of i.i.d. random variables.
Let $\tau_n$ be a simple random walk on $\integers^d$ independent of $\xi$s. We are interested
in $\DS S_N=\sum_{n=1}^{N} \xi_{\tau_n}.$ This model could be put in the present framework as
follows. Let $X$ be a set of sequences  $\{v_n\}_{n\in \integers},$ 
where $v_n\in \{\pm e_1, \pm e_{2}, \dots \pm e_d\}$ where $e_j$ are basis vectors in $\integers^d,$
$\mu$ is the Bernoulli measure with $\Prob(v=\pm e_j)=\frac{1}{2d}$ for all $j\in 1, \dots, d,$
$Y$ is the space of sequences $\{\xi_z\}_{z\in\integers^d}$, $\nu$ is the product of distribution
functions of $\xi$, $f$ and $G_t$ are shifts and $\tau(\{v\})=v_0.$ For random walks in random scenery,
Theorem \ref{ThCLT1} is due to \cite{KS}, and Theorem \ref{ThCLT2}  are due
to \cite{Bolt}. 
The results of \cite{KS} and \cite{Bolt} are extended to more general actions in the fiber
(still assuming the random walk in the base) in \cite{CC17, CC23}.

In the context of dynamical systems, 
Theorem \ref{ThCLT1} was proven in \cite{LB} under the assumption that we have a good 
rate of convergence in the Central Limit Theorem for $f$. In the present paper we follow a
method of \cite{Bolt} which seems more flexible and allows a larger class of base systems.

We also note that if $d\geq 3$ or $\tau$ has a drift, one has a classical 
CLT (that is, $\dfrac{S_N-\EXP(S_N)}{\sqrt{N}}$ converges to a Gaussian distributions). These
results are proven in \cite{DDKN, DDKN2}.
The paper \cite{DDKN} also studies mixing properties of 
$(T, T^{-1})$ transformation defined by \eqref{DefTTInv}. In particular \cite{DDKN} shows that in many cases the mixing of the whole product $F$ is 
exponential. 
We observe that in the case where
$F$ enjoys multiple exponential mixing,
 one can obtain the classical CLT by applying the results of \cite{BG}.
The CLT of Bj\"orklund and Gorodnik \cite{BG} also plays a key role in our proof and
we review it in Section \ref{ScProduct}.

In the case where the base map $f$ is uniformly hyperbolic the skew product map $F$ 
is partially hyperbolic. It is possible that generic partially hyperbolic maps enjoy strong statistical properties 
including the multiple exponential mixing and the Central Limit Thereom. However, such a result
seems currently beyond reach
(some special cases are considered in \cite{AGT, CL}).
In present $(T, T^{-1})$ setting the exotic limit theorems such as
our Theorems \ref{ThCLT1} and \ref{ThCLT2} require the zero drift assumption, which is a codimension 
$d$ condition.


\section{Bj\"orklund - Gorodnik CLT}
\label{ScProduct}

In order to prove our results we 
use the strategy of \cite{Bolt} replacing the Feller Lindenberg CLT for iid random variables by
a CLT for exponentially mixing systems due to \cite{BG}.
More precisely we need the following fact.

\begin{proposition}
\label{PrBG}
 Let $\fm_N$, $N \in \mathbb N$ be a sequence of measures on $\reals^d$
 and let $\{ A_{t,N}\}_{t \in \reals^d, N \in \mathbb N}$
 be a family of real valued functions on $Y$ so that
 $\|A_{t,N}\|_{C^1(Y)}$ is uniformly bounded and $\nu(A_{t,N})\equiv 0$.
 Set $\cS_N := \int_{\reals^d} A_{t,N}(G_t y) \d\fm_N(t).$
Suppose that
\smallskip

(a) $\DS \lim_{N\to\infty} \|\fm_N\|=\infty$  where 
$\DS \|\fm\|= \fm(\reals^d).$

(b) For each $r\in \naturals,$ $r\geq 3$ for each $K$
$$ \int \fm_N^{r-1} [B(t, K\ln \|\fm_N\|)] \d\fm_N(t)=0, $$
where $B(t,R)\subset \reals^d$ is the ball of radius $R$ around $t \in \reals^d$.

(c) $\DS \lim_{N\to \infty} V_N=\sigma^2$ where
$$ V_N:=\int\cS_N^2(y)\d\nu(y)= \iiint A_{t_1, N} (G_{t_1} y) A_{t_2, N} (G_{t_2} y) \d\fm_N (t_1) \d \fm_N(t_2) \d\nu(y).$$

Then, as $N \to \infty$, $\cS_N$ converges weakly to the normal 
distribution with zero mean and variance $\sigma^2.$
\end{proposition}

This proposition is proven in \cite{BG} in case $A_{t,T}$ does not depend on $t$ and $T$,
however the proof 
easily extends to the case of $t,T$-dependent $A$, see \cite{survey}.


\section{Dimension two}
\subsection{Reduction to quenched CLT}
Here we prove Theorem \ref{ThCLT2}.

Consider a function $\tH$ satisfying 
\begin{equation}
\label{ZMeanFiber}
\int \tH(x, y)\d\nu(y)=0
\end{equation}
 for all $x\in X.$

Given $x\in X$ define the measures $\fm_N$ 
on $\reals^2$
by
\begin{equation}\label{eq:d=2fm} \fm_N(x)=\frac{1}{\sqrt{N \ln N}} \sum_{n=0}^{N-1}  \delta_{\tau_n(x)}
\end{equation}
and functions
$A_{t,N,x}(y)$ by
\begin{equation}\label{eq:d=2A} 
A_{t,N,x}(y)=\frac{1}{ \Card ( n\leq N: \tau_n(x)=t)} 
\sum_{n\leq N: \tau_n(x)=t} \tH({ f^n}x,y). 
\end{equation}

\begin{proposition}
\label{LmRandMes3}
Under the assumptions of Theorem \ref{ThCLT2},
there exists $\sigma^2$
and subsets $X_N\subset X$ such that
$\DS \lim_{N\to\infty} \mu(X_N)=1$ and for any sequence $x_N\in X_N$
the measures $\{\fm_N(x_N)\}$ satisfy the conditions of Proposition \ref{PrBG}.
\end{proposition}

The proposition will be proven later. Now we shall show how to obtain 
Theorem~\ref{ThCLT2} from the
proposition.

\begin{proof}[Proof of Theorem \ref{ThCLT2} assuming
Proposition \ref{LmRandMes3}]
Split 
\begin{equation}
\label{BaseFiber}
H(x,y)=\tH(x,y)+\brH(x)\quad\text{where}\quad \brH(x)=\int H(x,y) \d\nu(y).
\end{equation}
 
Denote
$$
\cS_N(x,y) = \frac{1}{\sqrt{N \ln N}} \sum_{n=0}^{N-1} \tH(F^n(x,y)).
$$
Note that $\tH$ satisfies \eqref{ZMeanFiber} and hence, 
by Proposition \ref{LmRandMes3},
$\cS_N(x,y)$ is asymptotically normal and is asymptotically independent 
from $x$. 
Finally  the contribution of $\brH$ is negligible. Indeed, the 
CLT for smooth observables gives
$\DS
\frac{1}{\sqrt{N \ln N}} \sum_{n=0}^{N-1} \brH(f^n(x)) 
\Rightarrow 0.
$
This completes the proof of the theorem (modulo Proposition \ref{LmRandMes3}).
\end{proof}



The remaining part of this section is devoted 
to the proof of Proposition \ref {LmRandMes3}. It suffices to
verify the conditions of Proposition \ref{PrBG}.

{\bf Property (a)} is clear, since $\|\fm_N(x)\|=\sqrt{N/\ln N}.$

Verifying properties (b) and (c) requires longer
computations that are presented in 
\S\ref{secdim2propb} and \S\ref{secdim2propc}.

\subsection{Property (b)}
\label{secdim2propb}

Let 
\begin{equation}\label{eq:zz} X_{K, N}=\left\{x: \Card(n:|n|<N \text{ and }\|\tau_n(x)\|\leq K\ln N)\geq 
N^{1/5}\right\}.
\end{equation}
\begin{lemma}\label{LmQuarter}
If for each $ K$, $\DS \lim_{N\to\infty} N \mu(X_{K, N})=0$, then 
there are sets $\hX_N$ such that for all $x_N \in \hX_N$ the measures
$\fm_N(x_N)$ satisfy property (b) and $\mu(\hX_N) \to 1$
\end{lemma}

\begin{proof}
Given $K$ let 
$\DS \hX_N=\{x: f^n x\not\in X_{K, N}\text{ for } n<N\}. $
By the assumption of the lemma, $\mu(\hX_N)\to 1.$ 
On the other hand, for $x\in \hX_N$,
$$ \int \fm_N^{r-1}(x) [B (t, K\ln N)] \d \fm_N(x)=
\frac{1}{(N \ln N)^{r/2}} 
\sum_{n=0}^{N-1} \Card^{r-1} (j<N:\; \|\tau_j-\tau_n\|\leq K \ln N)$$
$$\leq N^{ -\frac{r}{2}+1 + \frac{r-1}{5}} \leq N^{-1/10},$$
where the first inequality holds since $x\in \hX_N$
and the second one holds because $r \geq 3$. 
Since $K$ is arbitrary, the result follows. 
\end{proof}

Let 
\begin{equation}
\label{DefLT}
\ell(x, t, N)=\Card(n \leq N:  |\tau_n(x)-t|\leq 1).
\end{equation}

\begin{lemma}
\label{LmLocTime} 
For each $p$ there is a constant $C_p$ such that
for each $t\in \reals^d$ and $N\in \naturals$,
\begin{equation}
\label{LocTimeMom3}
 \mu\left(\ell^p(\cdot, t,  N)\right)\leq C_p \ln^p N. 
\end{equation} 
\end{lemma}

\begin{proof}
Since $\DS \ell(x, t, N)=\sum_{n<N} \one_{B(t, 1)}(\tau_n(x))$ we have
$$ \mu\left(\ell^p(\cdot, t, n)\right)
=\sum_{n_1, \dots n_p} \mu\left(x: |\tau_{n_j}(x)-t|\leq 1\text{ for } j=1,\dots p
\right)$$
$$\leq
\sum_{n_1, \dots n_p} K_p \frac{1}{n_1} \prod_{j=2}^p \left(\frac{1}{n_j-n_{j-1}+1}\right),
$$
where the last step uses the anticoncentration large deviation bound.
\end{proof}

Lemma \ref{LmLocTime} and Markov inequality imply that for each $\eps, t, p$ we have
$$ \mu\left(x:  \ell(x, t, N)\geq \frac{N^{1/5}}{(K \ln N)^2} 
\right)
\leq \frac{C_p (K\ln N)^{2p} }{N^{ p/5}}$$
for $N$ large enough.
It follows that
$$ \mu\left(x:  \exists t: \|t\|\leq K\ln N\text{ and }\ell(x, t, N)\geq 
\frac{N^{1/5}}{ (K \ln N)^2} 
\right)\leq  \frac{C_p (K\ln N)^{2p+2} }{N^{p/5}}. $$
Taking $p=6$ we verify the conditions of Lemma \ref{LmQuarter}.


\subsection{Property (c)}
\label{secdim2propc}

Theorem 4.7 in \cite{DDKN} implies that
\begin{equation}
\label{Eq1Cor2}
\int \tH(x, y) \tH(F^k(x,y))
\d\mu(x) \d\nu(y)=\frac{\fg(0)}{k} \varsigma_2^2+o\left(\frac{1}{k}\right),
\end{equation}
where
\begin{equation}
\label{eq:sigma2def}
\varsigma_2^2 = \iiint_{- \infty}^{\infty} \int 
\tilde H (x_1,y) \tilde H (x_2,G_t(y))
 \d \nu(y) \d t \d\mu(x_1) \d \mu(x_2)
\end{equation}
and
  $\fg$ is the limit Gaussian density of $\tau$ (that is, $\tau_N/\sqrt{N}$ converges as
  $N\to\infty$ to the normal distribution with density $\fg$).
We note that the integral \eqref{eq:sigma2def} 
converges by the exponential mixing of $G$ and \eqref{ZMeanFiber}. 
Furthermore, $\varsigma_2^2 \geq 0$,
which can be seen from the following formula (whose proof
is standard and so is omitted)
$$
\varsigma_2^2 = \lim_{T \to \infty} \frac{1}{T}
\int \left[ \int  
\int_0^T \tH(x,G_t(y))
\d t \d\mu \right]^2 \d\nu.
$$

Defining 
$$ V_N=\frac{1}{N \ln N} \int 
\left[ \sum_{n=0}^{N-1} \tH(F^n(x,y))\right]^2
\d\nu(y) \, ,$$ 
we 
compute
\begin{align*}
    V_N &= \frac{1}{N \ln N}
    \sum_{n_1 = 0}^{N - 1}
    \sum_{n_2 = 0}^{N - 1}
    \int  \tH(F^{n_1}(x,y)) \tH(F^{n_2}(x,y)) 
    \d \nu(y)\\
    &= \frac{1}{N \ln N}
    \sum_{k = 0}^{N - 1}
    (N-k) (1 + \one_{k \neq 0})
    \int  \tH(x,y) \tH(F^{k}(x,y)) 
    \d \nu(y) \, ,
\end{align*}
whence from
\eqref{Eq1Cor2} we obtain
\begin{equation}
\label{DefSigma2}
 \Sigma^2:=\lim_{N\to\infty} \mu\left(V_N(x)\right)=
 2 \fg(0) \iiiint \tH(x, y) \tH(\brx, G_t y) \d\mu(x) \d\mu(\brx) \d\nu(y) \d t.
 \end{equation}
 
By Chebyshev's inequality, to establish {\em property (c)}, it suffices to show that 
\begin{equation}
\label{VarVar}
 \lim_{N\to\infty} \mu(V_N^2)=\Sigma^4. 
 \end{equation}
Note that 
$$\mu(V_N^2)=\frac{1}{N^2 \ln^2 N} \sum_{n_1, n_2, n_3, n_4} 
\mu(\sigma_{n_1, n_2} \sigma_{n_3, n_4}),$$
where
$$ \sigma_{n_1, n_2}(x)=\int H(f^{n_1} x, G_{\tau_{n_1}x} y) 
H(f^{n_2}x, G_{\tau_{n_2}(x)} y) \d\nu(y). $$
Fix a large $L$ and let $\bk=\bk(n_1, n_2, n_3, n_4)$ be the number of indices
$j$ such that $|n_i-n_j|\geq L$ for all $i\neq j.$ The number of terms 
where $\bk\leq 1$ is $O(L^2 N^2)$. Using
the trivial bound 
$|\mu(\sigma_{n_1, n_2} \sigma_{n_3, n_4})|
\leq \| \tH\|_{\infty}^2$, we see that the
contribution of terms with $\bk \leq 1$
is  $O(N^2 L^2)$ which is negligible.

Next, we consider the contribution of
the terms with $\bk=4$. 
{Without loss of generality, we will assume that 
\begin{equation}
\label{eq:combcond}
n_1<n_2, \quad n_3<n_4, \quad n_1<n_3.
\end{equation}
 We  distniguish $3$ cases based on  the relative position of $n_2$ with respect to $n_3$ and $n_4.$
\\

{\bf Case 1}.
We claim that
 for each $\eps>0$ there is $L$ such that if 
\begin{equation}
\label{eq:MMLLTappcase1}
n_1 < n_1 + L \leq n_2 < n_2 + L \leq n_3 < n_3 + L \leq n_4,
\end{equation}
then
\begin{equation}\label{eq:b-12} 
\left|\mu\left(\sigma_{n_1, n_2} \sigma_{n_3, n_4}\right)-\frac{\Sigma^4}{4|n_2-n_1| |n_4-n_3|}
\right|
\leq \frac{\eps}{|n_2-n_1| |n_4-n_3|}.
\end{equation} 
}

The proof of \eqref{eq:b-12} is based on
\cite{DDKN}. First, by \cite[Lemma 3.3]{DDKN},
we can assume that 
$\tH(x,y) = A(x) B(y)$ and without loss of generality we assume
$\nu(B)$ = 0. Then we have
$$
\mu\left(\sigma_{n_1, n_2} \sigma_{n_3, n_4}\right)
= 
\int \left[ \prod_{i=1}^4A(f^{n_i}x)\right]
\rho(\tau_{n_2-n_1}(f^{n_1}x))
\rho(\tau_{n_3-n_4}(f^{n_4}x))\d\mu(x),
$$
where
$\rho(t)=\int B(y)B(G_ty)\d\nu(y)$.
Note that for $\tH$ as above \eqref{VarVar} simplifies to
\begin{equation}
\label{eq:Sigmasq}
    \Sigma^2 = 2 \fg(0) \mu(A)^2 \int \rho(t) \d t.
\end{equation}
The remaining part of the proof of  \eqref{eq:b-12}
closely follows the lines of 
\cite[Theorems 4.6, 4.7]{DDKN} and so we only give a sketch.
Decompose $\reals^2$ into a countable
disjoint family of small squares $\cC_i$. Let $z_i$ be
the center of square $\cC_i$. Then
$$
\mu\left(\sigma_{n_1, n_2} \sigma_{n_3, n_4}\right)
\approx \sum_{i,j} S_{i,j},
$$
where
$$S_{i,j} = \rho(z_i) \rho(z_j)
\int \prod_{k=1}^4A(f^{n_k}x)
\one_{\cC_i}(\tau_{n_2-n_1}(f^{n_1}x))
\one_{\cC_j}(\tau_{n_3-n_4}(f^{n_4}x))\d\mu(x).
$$
Fixing $i,j$, letting $L \to \infty$ and 
{using the MMLLT and \eqref{eq:MMLLTappcase1},
we find
$$
S_{i,j}
\approx 
\frac{1}{|n_2 - n_1|} \frac{1}{|n_4 - n_3|}
\int_{\cC_i} \rho(t) \d t \int_{\cC_j} \rho(t) \d t
\fg(0)^2 [\mu(A)]^4.
$$
Summing over $i,j$ and using \eqref{eq:Sigmasq}, we obtain
\eqref{eq:b-12}.\\

{\bf Case 2.}
We claim that
\begin{equation}
\label{eq:MMLLTappcase2}
\sum_{(n_1,...,n_4): n_1 <n_3<n_4<n_2} 
\mu\left(\sigma_{n_1, n_2} \sigma_{n_3, n_4}\right) \leq
C N^2 \ln N.
\end{equation}
To prove \eqref{eq:MMLLTappcase2}
take
$(n_1,...,n_4)$   as in \eqref{eq:MMLLTappcase2}
and write
$$
m_0 = 0, \quad m_1 = n_3 - n_1, \quad  m_2 = n_4 - n_1, \quad m_3 = n_2 - n_1.
$$
Then we have
$$
\mu\left(\sigma_{n_1, n_2} \sigma_{n_3, n_4}\right) = 
\int \left[\prod_{i=0}^3A(f^{m_i}(x))\right] \rho(\tau_{m_3}(x))
\rho(\tau_{m_2 - m_1}(f^{m_1}x)) \d \mu(x).
$$
Decompose $\mathbb R^2$ into unit boxes $\cC_i$ with center $z_i$. Then, we find
$$
 \left|\mu\left(\sigma_{n_1, n_2} \sigma_{n_3, n_4}\right)\right| 
\leq C
\sum_{i_1,i_2,i_3} \tilde \rho(z_{i_3}) \tilde \rho(z_{i_2} - z_{i_1}) 
\int {\prod_{j=1}^3} \one_{\tau_{m_j} \in { \cC_{i_j}}}  \d \mu(x),
$$
where
$\DS 
\tilde \rho(t) = \sup_{t' : \| t' - t \| \leq 1}  \left|\rho(t)\right|.
$
Applying \eqref{eq:ACLD}, we find
\begin{align}
 \left|\mu\left(\sigma_{n_1, n_2} \sigma_{n_3, n_4}\right)\right|
 \leq & C
\sum_{i_1,i_2,i_3} \frac{ \tilde \rho(z_{ i_3}) \tilde \rho(z_{i_2} - z_{i_1}) }
{m_1(m_2 - m_1)(m_3 - m_2)}
\Theta\left(\max_{1 \leq j\leq 3} \frac{\| z_{i_j} - z_{i_{j-1}}\| }{\sqrt{m_j - m_{j-1}}} \right)
\label{eq:case2step1} \\
 & =: \sum_{i_1,i_2,i_3} \cS_{i_1,i_2,i_3}(m_1,m_2,m_3).
\nonumber
\end{align}
Let us consider the special case when $i_2 = i_1$ and $z_{i_3} = 0$ (without loss of generality we say 
that in this case $i_3 = 0$). Then we have
\begin{equation}
\label{eq:case2step2}
 \sum_{i_1} \cS_{i_1,i_1,0} (m_1,m_2,m_3)
 \leq 
\end{equation}
$$ 
 C
\sum_{i_1} 
\frac{1}{m_1}
\frac{1}{m_2 - m_1}
\frac{1}{m_3 - m_2}
\Theta \left( \frac{\|z_{i_1}\|}{\min \{ \sqrt{m_1}, \sqrt{m_3 - m_2}\}} \right).
$$
Note that the last expression is symmetric in $m_1 = n_3 - n_1$ and $m_3 - m_2 = n_2 - n_4$.
Thus without loss of generality we can assume $n_3 - n_1 \leq n_2 - n_4$ (indeed, the multiplier
$2$ can be incorporated into the constant $C$). Denoting $a = n_3 - n_1$, $b = n_4 - n_3$,
$c = n_2 - n_4$, we obtain
$$
\sum_{(n_1,...,n_4): n_1 <n_3<n_4<n_2 < N} 
\sum_{i_1} \cS_{i_1,i_1,0}(m_1,m_2,m_3)
\leq C
N \sum_i \sum_{a = 1}^{N} \sum_{b=1}^N \sum_{c=a}^N \frac{1}{abc} \Theta
\left( \frac{\|z_{i}\|}{\sqrt a} \right).
$$
Note that the multiplier $N$ accounts for all choices of $n_1$. Thus
$$
\sum_{(n_1,...,n_4): n_1 <n_3<n_4<n_2 < N} 
\sum_{i_1} \cS_{i_1,i_1,0}(m_1,m_2,m_3)
 \leq C
N \ln N \sum_i \sum_{a = 1}^{N} \sum_{c=a}^N \frac{1}{ac} \Theta
\left( \frac{\|z_{i}\|}{\sqrt a} \right)
$$$$
\leq C N \ln N  \int_{\reals^2} \int_{1}^{N} \int_{a}^N 
\frac{1}{ac} \Theta
\left( \frac{\|z\|}{\sqrt a} \right) \d c \d a\d z
\leq C N \ln N  \int_{1}^{N} \int_{a}^N 
\frac{1}{c} \d c \d a
\leq C N^2 \ln N
$$
where in the third inequality we used that 
$\DS {\int_1^{\infty} \Theta(r) r dr} < \infty$ and the
last inequality follows from direct integration. 
{Thus we have  verified
that the terms corresponding to $i_1 = i_2$ and $i_3 = 0$ are in $O(N^2 \ln N)$.

To complete the proof  of 
\eqref{eq:MMLLTappcase2}, we need to estimate the contribution of the sum corresponding to all
$i_2$ and $i_3$. To this end, we will
distinguish two cases based on whether the following assumption holds or not:
\begin{equation}
\label{eq:i2i3cluster}
\| z_{i_2} - z_{i_1} \| < \frac{1}{4} \|z_{i_1} \|, \text{ and }
\| z_{i_3} - z_{i_2} \| < \frac{1}{4} \|z_{i_1} \|.
\end{equation}
For any given fixed $i_1$, let 
$\widehat \sum_{i_2,i_3} $ denote the sum over $i_2, i_3$ that do not satisfy \eqref{eq:i2i3cluster}
and let 
$\widetilde \sum_{i_2,i_3} $ denote the sum for $i_2, i_3$ satisfying \eqref{eq:i2i3cluster}.

First, we assume that \eqref{eq:i2i3cluster} fails. Then, we have
\begin{multline*}
\sum_{i_1} \widehat{\sum_{i_2,i_3}} \cS_{i_1,i_2,i_3} (m_1,m_2,m_3) 
 \leq C
\sum_{i_1}  \widehat{\sum_{i_2,i_3}} \tilde \rho(z_{i_2} - z_{i_1}) \tilde \rho(z_{i_3}) 
\frac{1}{m_1}
\frac{1}{m_2 - m_1}
\frac{1}{m_3 - m_2} \nonumber \\
\times \left[ \Theta \left( \frac{\|z_{i_1}\|}{4 \min \{ \sqrt{m_1}, \sqrt{m_2 - m_1}\}} \right)
+ 
 \Theta \left( \frac{\|z_{i_1}\|}{4 \min \{ \sqrt{m_1},  \sqrt{m_3 - m_2}\}} \right)
\right].
\nonumber 
\end{multline*}
Now we can perform the summation with respect to 
$i_2,i_3$ first. 
 Denote $a = m_1$, $b = m_2 - m_1$, $c = m_3 - m_2$.
By the exponential decay of $\tilde \rho$ 
we get
\begin{multline*}
\sum_{i_1} \widehat {\sum_{i_2,i_3}} \cS_{i_1,i_2,i_3} (m_1,m_2,m_3) 
\leq C \sum_{i_1} 
\frac{1}{abc}
\left[ \Theta\! \left( \frac{\|z_{i_1}\|}{4 \min \{ \sqrt{a}, \sqrt{b}\}} \right)
\!+\! 
 \Theta\!\left( \frac{\|z_{i_1}\|}{4 \min \{ \sqrt{a},  \sqrt{c}\}} \right)
\right]\!.
\end{multline*}
Summing
the last displayed formula for $n_1,n_2,n_3,n_4$, we obtain a term $O(N^2 \ln N)$ exactly as in 
the estimate of \eqref{eq:case2step2}.

Now we assume that \eqref{eq:i2i3cluster} holds.
Then observing that
$
\| z_{i_3}\| \geq \| z_{i_1}\| / 2,
$
and using the exponential decay of $\tilde \rho$, we find that
$$
\tilde \rho(z_{i_2} - z_{i_1}) \tilde \rho(z_{i_3})
\leq C e^{-c \| z_{i_1} \|}
$$
with a fixed constant $c$.
Thus we conclude
 \begin{multline*}
\sum_{m_1 \leq m_2 \leq m_3 \leq N}
\sum_{i_1} \widetilde{\sum_{i_2,i_3}} \cS_{i_1,i_2,i_3}  \\
\leq C
\sum_{m_1 \leq m_2 \leq m_3 \leq N}
 \frac{1}{m_1}
\frac{1}{m_2 - m_1}
\frac{1}{m_3 - m_2}
\sum_{i_1} \widetilde {\sum_{i_2,i_3} }
e^{-c \| z_{i_1} \|}
\leq C  \ln^3 N.
\end{multline*}
Summing over all choices of $n_1$, we obtain a
term in $O(N  \ln^3N)$. }
This completes the proof of 
\eqref{eq:MMLLTappcase2}.\\

{\bf Case 3.}
We claim that
\begin{equation}
\label{eq:MMLLTappcase3}
\sum_{(n_1,...,n_4): n_1 <n_3<n_2<n_4} 
\mu\left(\sigma_{n_1, n_2} \sigma_{n_3, n_4}\right) \leq
C N^2 \ln N.
\end{equation}

The proof of \eqref{eq:MMLLTappcase3} is similar to that of \eqref{eq:MMLLTappcase2}.
Indeed, the right hand side of \eqref{eq:case2step1} is now replaced by
$$
C
\sum_{i_1,i_2,i_3} \frac{\tilde \rho(z_{i_2}) \tilde \rho(z_{i_3} - z_{i_1}) }
{m_1(m_2 - m_1)
(m_3 - m_2)}
\Theta\left(\max_j \frac{\| z_{i_j} - z_{i_{j-1}}\| }{\sqrt{m_j - m_{j-1}}} \right).
$$
{We consider the special case $i_2 = 0, i_3 = i_1$ first. Then \eqref{eq:case2step2}
is replaced by 
 \begin{multline}
\label{eq:case2step2v2}
\sum_{i_1} \cS_{i_1,0,i_1} (m_1,m_2,m_3) \\
 \leq C
\sum_{i_1} 
\frac{1}{m_1}
\frac{1}{m_2 - m_1}
\frac{1}{m_3 - m_2}
\Theta \left( \frac{\|z_{i_1}\|}{\min \{ \sqrt{m_1}, \sqrt{m_2 - m_1}, \sqrt{m_3 - m_2}\}} \right).
\end{multline}
which is bounded by the right hand side of \eqref{eq:case2step2} (perhaps with a 
different constant $C$) and hence is in $O(N^2 \ln N)$.
The contribution of general $i_2$, $i_3$ can be bounded as before. \eqref{eq:MMLLTappcase3} follows.}

\smallskip
The upshot of the above three cases is that
the leading term only comes from indices $n_1,...,n_4$
satisfying \eqref{eq:MMLLTappcase1}.
Summing $\mu(\sigma_{n_1,n_2} \sigma_{n_3,n_4})$
for all indices $n_1,...,n_4$
satisfying \eqref{eq:MMLLTappcase1}, we obtain
$N^2 \ln^2 N [\Sigma^4 /8 + o_{L \to \infty}(1)]$ 
 where we have used 
\eqref{eq:b-12} and the fact that
$\DS \sum_{n_1<n_3<N} \ln (n_3-n_1) \ln (N-n_3)=\frac{N^2 \ln^2 N}{2} (1+o_{N\to\infty}(1)). $

\smallskip
Next we claim that
\begin{equation}
\label{D2VarSum}
\sum_{n_1,...,n_4: \bk = 4} \mu(\sigma_{n_1, n_2} \sigma_{n_3,n_4})
= [ \Sigma^2  + o_{L \to \infty}(1)] N^2 \ln^2 N.
\end{equation}

 Indeed similarly to the case of \eqref{eq:MMLLTappcase1} considered above we see that the main 
contribution to our sums comes from the terms where the pairs $(n_1, n_2)$ and $(n_3, n_4)$ are not 
intertwined. Note that if we choose a random permutation of $(n_1, n_2, n_3, n_4)$ then the
probability of non intertwining is 1/3 (since the second element should come from the same 
pair as the first one). Therefore there are 24/3=8 permutations
\footnote{These eight permutations corresponds to all possible choices of signs
($*\in \{\leq, \geq\}$)
 in the 
inequalities 
$$n_1* n_2, \quad n_3*n_4,\quad  \min(n_1, n_2)* \min(n_3, n_4)$$}
 contributing to our sum
which explains the additional factor of 8 in \eqref{D2VarSum}.

}

To complete the proof of property (c),
it remains to verify that the contribution 
of terms with $\bk = 2,3$
is negligible. First note that $\bk = 3$ is impossible
by the definition of $\bk$. Finally, if 
$\bk = 2$, the we can repeat a simplified version of the
proof for $\bk = 4$ using the MLLT instead
of the MMLLT. For example, let us consider the
case $n_1\leq n_2\leq n_3\leq n_4$.
Since $\bk = 2$, we must have  
$|n_2 - n_1| \geq L$ or $|n_4 - n_3| \geq L$. Without loss
of generality assume $|n_4 - n_3| \geq L$. Then, using the 
fact that due to the exponential mixing of $G$,
$ \sigma_{n_3, n_4}\leq C e^{-c \ell} $ on the set $\|\tau_{n_4}-\tau_{n_3}\|\in [\ell, \ell+1]$, we obtain from the MLLT 
that
$$\mu\left(\sigma_{n_1, n_2} \sigma_{n_3, n_4} \right)=O\left(\frac{1}{ |n_4-n_3|}\right). $$
Hence the total contribution 
terms with $\bk = 2$ is in $O\left(L N^2 \ln N\right)$.
This completes the proof of 
property (c).
We have finished the proof of Proposition
\ref{LmRandMes3} and hence the proof of Theorem \ref{ThCLT2}.


\section{Dimension one}
\label{ScD=1}
\subsection{Reduction to the limit theorem for the variance.}
Here we prove Theorem~\ref{ThCLT1}.

We use the decomposition \eqref{BaseFiber}. Since $f$
satisfies the CLT for smooth observables, $\DS \frac{1}{\sqrt N} \sum_{n=1}^N \bar H (f^n(x)) $
converges weakly, and so this sum is negligible after rescaling by $N^{3/4}$. It remains to study the ergodic 
sum of $\tH$.

Following the ideas of the previous section, we let
\begin{equation}
\label{DefMND1}
 \fm_N(x)=\frac{1}{\sqrt{V_N(x)}} \sum_{n=0}^{N-1}  \delta_{\tau_n(x)},
\end{equation}
$$A_{t,N,x}(y)=\frac{1}{\mbox{Card}({n\leq N: \tau_n(x)=t})} 
\sum_{n\leq N: \tau_n(x)=t} \tH(f^n(x),y),
$$
where
$$ V_N(x)=\int S_N^2(x,y) \d\nu(y), \quad 
S_N(x,y) = \sum_{n=0}^{N-1} \tH \circ F^n(x,y).
$$
Thus in the notation of Proposition \ref{PrBG}, we have 
$\DS \cS_N=\frac{S_N}{\sqrt{V_N(x)}}.$

In contrast with Theorem \ref{ThCLT2}, $V_N(x)$ does not satisfy
a weak law of large numbers. Instead we have
\begin{proposition}
\label{prop:localtime} 
There is a constant $\Lambda$ so that
$\frac{V_N}{ \Lambda N^{3/2}}$ converges in law as $N\to\infty$ to the random variable $\cL^2$ given by 
\eqref{DefL}. 
\end{proposition}  

Before proving Proposition \ref{prop:localtime}, let us  use it to derive Theorem \ref{ThCLT1}.
We start with the following analogue of  Proposition  \ref{LmRandMes3}:

\begin{lemma}
\label{LmRandMes1}
Under the assumptions of Theorems \ref{ThCLT1},
there are subsets $X_N\subset X$ such that
$\DS \lim_{N\to\infty} \mu(X_N)=1$ and for any sequence $x_N\in X_N$
the measures $\{\fm_N(x_N)\}$ satisfy the conditions of Proposition \ref{PrBG}
with $\sigma = 1$.
\end{lemma}

\begin{proof}[Proof of Lemma \ref{LmRandMes1} assuming Proposition \ref{prop:localtime}]

Clearly, $\| \fm_N \| \to \infty$ for a large set of $x$'s and so (a) holds.
Recalling that $\cS_N (x,y) =  S_N(x,y) / \sqrt{V_N(x)}$, property (c) holds
trivially for all $x \in X$. It remains to show
property (b).

Note that $\cL^2$ is non-negative and non-atomic at zero, i.e.
for any $\eps > 0$ there is $\eps' > 0$ so that $\mathbb P (\cL^2 < \eps') < \eps$
and so by Proposition \ref{prop:localtime}, 
$$ \lim_{N\to\infty} \mu(x: N^{1.49} < V_N(x) < N^{1.51}) =1.$$
Thus we can assume that $x$ is such that $N^{1.49} < V_N(x) < N^{1.51}$.
Next, we define
$$
\tilde \ell (x,t,N) = \frac{\Card (n \leq N: |\tau_n(x) - t| < 1)}{\sqrt N}.
$$
As in the proof of Lemma \ref{LmLocTime},
we find that for every $p \in \integers_+$, 
 every $t\in \reals$, and every $N\in\naturals$ we have
 $\mu(\tilde \ell^p(.,t,N)) < C_p$. 
Then choosing $p=1000$, the Markov inequality gives
$\mu(x: \tilde \ell (x,t,N) > N^{1/100}) < C N^{ -10}$. Thus we can assume that for all $t \in \integers$ with $|t| < 
N^{0.6}$,
$\tilde \ell (x,t,N) < N^{1/100}$ holds. By the anticoncentration large deviation bounds, we can also assume that 
$\DS \max_{n \leq N} |\tau_n(x)| \leq N^{0.6}$
since the set of points where this condition fails has negligible measure. In summary,
 the measure of the set of $x$'s satisfying that 
 $$
N^{1.49} < V_N(x) < N^{1.51} \quad\text{and}\quad
\tilde{\ell}(x,t, N)\leq \begin{cases} N^{0.01} & |t|\leq N^{0.6}, \\ 0 & |t|\geq N^{0.6}. \end{cases} $$
 tends to $1$ as $N\to\infty.$
For any such $x$, we have
$$ \int \fm_N^{r-1}(x) (t, K\ln N) \d \fm_N(x)=
\frac{1}{(V_N(x))^{r/2}} 
\sum_{n=0}^{N-1} \Card^{r-1} (j<N:\; |\tau_j-\tau_n|\leq K \ln \| \fm_N \|)$$
$$
\leq \frac{1}{(N^{1.49})^{r/2}} 
 \sum_{t\in \integers: |t|\leq N^{0.6}} \ell^r(x, t, N)\leq
2 N^{0.51r+0.6-0.745r}=o(1). 
$$
if $r \geq 3$. Property (b) and the lemma follow.
\end{proof}

\begin{proof}[Proof of Theorem \ref{ThCLT1} assuming Proposition \ref{prop:localtime}]
Recall that
$\DS
\cS_N(x,y) = \frac{S_N(x,y)}{ \sqrt{V_N(x)}}.
$
By Lemma \ref{LmRandMes1} and Proposition \ref{PrBG}, 
there is a sequence of subsets $X_N \subset X$ with $\mu(X_N) \to 1$
such that for every sequence $x_N \in X_N$, 
the distribution of 
$\cS_N(x,y)$ w.r.t. $\nu(y)$ converges to the standard normal (note that this time 
properties (a) and (c) are immediate). 

 In fact, we have the stronger statement that 
\begin{equation}
\label{ConvPair}
\left(\frac{V_N}{\Lambda N^{3/2}}, \cS_N\right)\Rightarrow 
(\cL^2, \cN)\quad\text{as}\quad N\to \infty,
\end{equation}
where $\cL$ is given by \eqref{DefL}, 
$\cN$ is standard normal and $\cL^2$ and $\cN$ are independent.
 This follows from Proposition \ref{prop:localtime}, Lemma \ref{LmRandMes1} and the asymptotic independence of $\cS_N$ and $V_N$ which comes from the fact that
$V_N$ depends only on $x$ while $\cS_N$ is asymptotically independent of $x$.
More precisely, let $\phi$ and $\psi$ be two continuous compactly supported test functions.
Then
$$\lim_{N\to\infty} \zeta\left(\phi\left(\frac{V_N}{\Lambda N^{3/2}}\right) \psi(\cS_N)\right)=
\lim_{N\to\infty} \zeta\left(\phi\left(\frac{V_N}{\Lambda N^{3/2}}\right) \psi(\cS_N) \one_{X_N} \right)
$$$$
=\lim_{N\to\infty} \int_{X_N} \left[\phi\left(\frac{V_N}{\Lambda N^{3/2}}\right) \int_Y \psi(\cS_N)\d\nu   
\right]
\d\mu
=\EXP(\psi(\cN)) \lim_{N\to\infty} \int_{X_N} \phi\left(\frac{V_N}{\Lambda N^{3/2}}\right) 
\d\mu
$$$$
=\EXP(\psi(\cN)) \lim_{N\to\infty} \int_{X} \phi\left(\frac{V_N}{\Lambda N^{3/2}}\right) 
\d\mu=\EXP(\psi(\cN)) \EXP(\phi(\cL^2)), 
$$
where the first  and the fourth equalities hold since $\mu(X_N)\to 1,$ 
the third inequality holds by Proposition~\ref{PrBG}, and the last equality holds
by Proposition \ref{prop:localtime}.
This proves \eqref{ConvPair}.

Denote $\Sigma = \sqrt \Lambda$. By \eqref{ConvPair} and
the continuous mapping theorem,
$$ \frac{S_N}{\Sigma N^{3/4}}=
\frac{\sqrt{V_N}}{\sqrt{\Lambda N^{3/2}}} \times \frac{\cS_N}{\sqrt{V_N}}
\Rightarrow \cL \cN 
$$
completing the proof of the Theorem \ref{ThCLT1}.
\end{proof}

\subsection{Proof of Proposition \ref{prop:localtime}}
We are going to use the following well known fact (see
 e.g. \cite{B} Chapter 1.7, Problem 4).
If $\mathcal X_n$ is a sequence of random variables so that for every $k \in \integers_+$
\begin{equation}
\label{eq:moments}
J_k = \lim_{n \to \infty} \mathbb E (\mathcal X_n^k) 
\end{equation} exists and 
\begin{equation}
\label{eq:analyticmgf}
\limsup_{k \to \infty} \left( \frac{J_k}{k!} \right)^{1/k} < \infty,
\end{equation}
then $\mathcal X_n$ converges weakly to a random variable $\mathcal X$. Furthermore, 
$\mathbb E(\mathcal X^k) = J_k$ for every $k \in \integers_+$
and $\mathcal X$ is uniquely characterized by its moments.

Now we explain
our strategy of proof of Proposition \ref{prop:localtime}
(a similar strategy was used in \cite{NSz}).
We prove that there is a constant
$\Lambda$
depending on the $(T,T^{-1})$ transformation $F$
and the observable $H$
so that 
$\mathcal X_N = V_N/(\Lambda N^{3/2})$ satisfies \eqref{eq:moments}
(with $\mathbb E$ meaning integral w.r.t. $\mu$)
and \eqref{eq:analyticmgf} with constants $J_k$ that do not depend on 
$F$ and $H$.
Consequently, there is a random variable $\mathcal X$ so that $\mathcal X_N = V_N/(\Lambda N^{3/2})$
converges weakly to $\mathcal X$ for any $(T,T^{-1})$ transformation
satisfying the assumptions of Theorem \ref{ThCLT1}. Recall from \S \ref{subsec:RWRS} that one such example is the one dimensional random walk in random scenery,
in which case by the result of \cite{KS}, 
$\mathcal X_N = V_N/(\Lambda N^{3/2})$
converges weakly to $\mathcal L^2 $. Thus $\mathcal X=\mathcal L^2$ and it
has to be the limit for all $(T,T^{-1})$ transformation
satisfying the assumptions of Theorem \ref{ThCLT1}. It remains to 
check \eqref{eq:moments}
and \eqref{eq:analyticmgf}.

Recall that $\tau$ satisfies the MMLLT with a Gaussian density 
$\mathfrak g$. Let $\varsigma_1$ be the standard deviation of this Gaussian random variable, that is
$\mathfrak g(z) = \varphi(z/\varsigma_1)/\varsigma_1$, where $\varphi$ is the standard Gaussian
density. 

Define
\begin{equation}
\label{defLambda}
\Lambda = \frac{\sqrt \pi \varsigma_2^2 }{\varsigma_1}=\sqrt{2} \pi \varsigma_2^2 \fg(0)
\end{equation}
where $\varsigma_2^2$ is given by \eqref{eq:sigma2def}.
 
\begin{lemma}
\label{lem:firstmoment}
For $k=1$, \eqref{eq:moments} holds with 
\begin{equation}
\label{eq:moment2}
J_1 = \frac{2}{ \sqrt \pi}
\int_{w \in \reals} \int_{0 < t_1 < t_2 < 1} \frac{1}{\sqrt{t_1}}
\frac{1}{\sqrt{t_2- t_1}}
\varphi \left( \frac{w}{\sqrt{ t_1}}\right) \varphi(0)\d t_1 \d t_2 \d w .
\end{equation}
\end{lemma}

We note that the above integral can be performed explicitly. Namely 
$\DS J_1 = \frac{4 \sqrt 2}{3 \sqrt \pi}$. 
 This shows that $\Lambda$ was chosen
correctly. Indeed, in case of the simple random walk in random 
scenery,
formula (1.2) in 
\cite{KS} shows that $J_1 = \frac{4 \sqrt 2 }{3 \sqrt \pi}$
while by the main result of \cite{KS}, 
the conclusion of Proposition
\ref{prop:localtime} holds for simple random walk in random scenery.
\begin{proof}[Proof of Lemma \ref{lem:firstmoment}]
Fix some $\eps >0$. We need to show that for $N$ large enough,
$$
\left| 
N^{-3/2} \Lambda^{-1}
\iint \left[ \sum_{n=1}^N \tH \circ F^n (x,y) \right]^2 \d \mu \d \nu - J_1 \right| <
\eps.
$$
In the following proof,
we will choose a small 
$\eta = \eta(\eps)$, a large $L = L(\eps)$ a small 
$\eta'\!\!=\!\! \eta' (\eps, \eta, L)\!\!<\!\! \eta$ and
finally $N = N(\eps, \eta,L, \eta')$. There will be finitely many restrictions on these parameters, 
we can choose the most strict one. 

Choose a partition of $X$ into sets $X_i, i\in I$
of diameter $< \eta$ and fix some elements $x_i \in X_i$.
We have
$$
\iint \left[ \sum_{n=1}^N \tH \circ F^n (x,y) \right]^2 \d \mu \d \nu =
$$
\begin{equation}
\label{eq:sum1dim1}
 \sum_{n_1, n_2 = 1}^N \sum_{i_1,i_2 \in I}
\sum_{z_1,z_2 \in \integers} \int 
p(x) \int \tH (f^{n_1}(x),G_{\tau_{n_1}(x)} (y)) 
\tH (f^{n_2}(x),G_{\tau_{n_2}(x)} (y)) \d\nu \d\mu ,
\end{equation}
where
$$
p(x) = p_{\eta,n_1,n_2,i_1,i_2,z_1,z_2}(x) = 
\one_{f^{n_1}x \in X_{i_1}}
\one_{f^{n_2}x \in X_{i_2}}
\one_{\tau_{n_1}(x) \in [z_1, z_1 + 1]\eta}
\one_{\tau_{n_2}(x) \in [z_2, z_2 + 1]\eta}.
$$
Split the sum in \eqref{eq:sum1dim1} as $T_1 + T_2$, where $T_1$
corresponds to the terms satisfying
\begin{itemize}
\item[(A1)]
$\DS \eta' N < \min\{ n_1, n_2\} < \min\{ n_1, n_2\} + \eta' N < \max \{ n_1, n_2\}
$
\item[(A2)]$|z_1| < L \sqrt N /\eta$
\item[(A3)] $|z_1 - z_2| < L / \eta$.
\end{itemize} and $T_2$ stands for the terms where 
 at least one of the conditions (A1)--(A3) 
is violated.
 Write $T_1 = \widetilde\sum (...)$.

We start by estimating $T_1$. Let us write $a_N \approx b_N$ if there are 
constants $\eta$, $L$ and $\eta'$ so that for $N$ large enough,
$|a_N -b_N | < \eps  \Lambda N^{3/2}  / 10$. 
We claim that
  \begin{equation}
\label{eq:T_1}
T_1 \approx 
\widetilde\sum 
 \int 
p(x) \d\mu  \int \tH (x_{i_1},G_{z_1 \eta} (y)) 
\tH (x_{i_2},G_{z_2 \eta} (y)) \d\nu .
\end{equation}
{Indeed, by the continuity of $H$ we can choose 
$\eta$ so small that the difference between the LHS and the RHS of 
\eqref{eq:T_1} does not exceed
$$ \frac{\eps  \Lambda}{1000} \widetilde\sum  \int
p_{\eta,n_1,n_2,i_1,i_2,z_1,z_2}(x)  \d\mu(x) ,$$
so 
\eqref{eq:T_1} follows from the 
anticoncentration large deviation bound.

Next,}
writing $m_1 = \min\{ n_1, n_2\} $, $m_2 = \max\{ n_1, n_2\} $
and using the MMLLT, we find
\begin{equation}
\label{T1-LLTSum}
T_1 \approx 
\frac{1}{\varsigma_1^2} \widetilde \sum \eta^2
\frac{1}{\sqrt{m_1}} \varphi \left( \frac{z_1 \eta}{\varsigma_1 \sqrt{m_1}}\right) 
\frac{1}{\sqrt{m_2 - m_1}} 
\varphi \left( \frac{(z_2 - z_1) \eta}{\varsigma_1 \sqrt{m_2 - m_1}}\right)
\mu(X_{i_1}) \mu(X_{i_2})\times
\end{equation}
$$
\int \tH (x_{i_1},G_{z_1 \eta} (y)) \tH (x_{i_2},G_{z_2 \eta} (y)) \d\nu .
$$
 Noting that  by (A1) and (A3),
$(z_2 - z_1) \eta / [\varsigma_1 \sqrt{m_2- m_1}] = O(N^{-1/2})$,
we see that $\DS \varphi \left( \frac{(z_2 - z_1) \eta}{\varsigma_1 \sqrt{m_2 - m_1}}\right)$
can be replaced by $\varphi(0)$ in \eqref{T1-LLTSum} without invalidating $\approx$.
Then the only 
term depending on $z_2$ is under the integral with respect to $\nu$. We
can approximate the sum over $z_2$ by a Riemann integral:
\begin{equation}
\label{Cor-SumInt}
\eta \sum_{z_2 = z_1 - L /\eta }^{z_1 + L /\eta } 
\int \tH (x_{i_1},G_{z_1 \eta} (y)) \tH (x_{i_2},G_{z_2 \eta} (y)) \d\nu =
\cI_{i_1,i_2}
+o_{ \eta\to 0, L\to \infty, N \to \infty} (1) 
\end{equation}
where
$$
\cI_{i_1,i_2}=
\int_{ \reals} \int \tH(x_{i_1}, y) \tH(x_{i_2}, G_t (y)) \d \nu \d t .
$$
Combining \eqref{T1-LLTSum} and \eqref{Cor-SumInt},
we conclude
\begin{eqnarray*}
T_1 &\approx& \frac{2}{\varsigma_1^2}
\left[\sum_{i_1,i_2} \mu(X_{i_1}) \mu(X_{i_2}) \cI_{i_1,i_2} \right] \times \\
&&\times
\left[  \sum_{\eta' N < m_1 < m_1 + \eta' N < m_2 <N}
\sum_{|z_1| < L \sqrt N/\eta} \eta
\frac{1}{\sqrt{m_1}} \varphi \left( \frac{z_1 \eta }{\varsigma_1 \sqrt{m_1}}\right) 
\frac{1}{\sqrt{m_2 - m_1}}  \varphi(0) \right],
\end{eqnarray*}
where the factor $2$ comes from taking into account both $n_1 < n_2$ and $n_2 < n_1$.
Further decreasing $\eta$ and increasing $L$ if necessary
and recalling \eqref{eq:sigma2def}, we can replace the 
first sum by $\varsigma_2^2$.

Now replacing  two Riemann sums with the corresponding Riemann integrals, 
we obtain 
$$
T_1 \approx
\frac{2 \varsigma_2^2}{\varsigma_1^2} N^{3/2}
\int_{|w| < L} \int_{\eta' < t_1 < t_1 + \eta' < t_2 < 1} \frac{1}{\sqrt{t_1}}
\frac{1}{\sqrt{t_2- t_1}}
\varphi \left( \frac{w}{\varsigma_1\sqrt{ t_1}}\right) \varphi(0) \d t \d w
\approx \frac{\varsigma_2^2}{\varsigma_1} 
 \sqrt \pi N^{3/2} J_1.
$$
In order to complete the proof of 
\eqref{eq:moment2}, it suffices to show that
for $L$ large and for $\eta, \eta'$ small
\begin{equation}
\label{eq:T_2}
|T_2 | < \eps N^{3/2} / 10.
\end{equation}
We will estimate the contribution of the terms where exactly one of the 
assumptions (A1)--(A3) fail (the cases when more than one fail are similar and easier).

Suppose that (A1) fails and, moreover, $m_1<\eta' N$
(the other cases are similar). Using the anticoncentration large deviation bounds instead of the
MMLLT, we obtain that the corresponding sum is bounded by
$$
\sum_{m_1=1}^{\eta' N} \sum_{m_2=0}^{N} \frac{1}{\sqrt{m_1}} \frac{1}{\sqrt{m_2 +1}}
\sum_{|z_1| < L \sqrt N /\eta} C(\eta, L) \leq \sqrt{\eta'} N C'(\eta, L) \sqrt N < \eps N^{3/2} /10
$$ 
if $\eta'$ is small enough.

Let us assume that (A2) fails. Then again using the anticoncentration large deviation bounds instead of the
MMLLT, we obtain that the corresponding sum is bounded by
$$
\sum_{m_1 = \eta' N}^N 
\sum_{m_2 = m_1 + \eta' N}^N
 \frac{1}{\sqrt{m_1}} \frac{1}{\sqrt{m_2 - m_1}}
\sum_{|z_1| \geq  L \sqrt N /\eta} C \eta 
\Theta(|z_1| \eta / \sqrt{m_1}) $$
$$
 \leq 
C N \sum_{|z_1| \geq  L \sqrt N /\eta} \eta 
\Theta(|z_1| \eta / \sqrt{N}) \leq C N^{3/2} \int_{L-1}^{\infty} \Theta (r) \d r
 < \eps N^{3/2} /10
$$ 
for $L$ large. 

Finally assume that (A3) fails. Then we use the exponential mixing of $G$ to derive
$$
\eta \sum_{z_2:  |z_2 - z_1| > L /\eta }
\int \tH (x_{i_1},G_{z_1 \eta} (y)) \tH (x_{i_2},G_{z_2 \eta} (y)) \d\nu  $$
$$
\leq 2
\int_{L}^{\infty} \int \tH(x_{i_1}, y) \tH(x_{i_2}, G_t (y)) \d \nu \d t 
\leq C e^{-cL}.
$$
Proceeding as in the case of $T_1$ we obtain that the sum corresponding to indices when (A3) is violated is 
bounded by $CN^{3/2} e^{-cL} < \eps N^{3/2} / 10$. 
This completes the proof of the lemma.
\end{proof}

\begin{lemma}
\label{lem:highermoment}
For any $k\geq 2$, \eqref{eq:moments} holds with 
\begin{equation}
\label{eq:momentk}
J_{k} = \left[\frac{2}{\sqrt \pi} \right]^k \int_{w_1,...,w_k \in \mathbb R}
\int_{0 < t_1 < ... < t_{2k} < 1}
\prod_{j=1}^{2k} \frac{1}{\sqrt{t_j - t_{j-1}}}
\sum_{v \in V} 
\prod_{j=1}^{2k} \varphi
\left( \frac{ w_{v(j)} - w_{v(j-1)}}{ \sqrt{t_j - t_{j-1}}}
\right) \d t \d w,
\end{equation}
where $V$ is the set of all two-to-one mappings from $\{1,2,...,2k\}$ to $\{1,2,...,k\}$
(that is, $|v^{-1}(l)| = 2$ for all $l=1,...,k$).
\end{lemma}

\begin{proof}[Proof of Lemma \ref{lem:highermoment}]
We follow the strategy of the proof of Lemma \ref{lem:firstmoment}.

Fix some $\eps >0$. We need to show that for $N$ large enough,
$$
\left| 
N^{-3k/2} \Lambda^{-k}
\int \left( \int \left[ \sum_{n=1}^N \tH \circ F^n (x,y) \right]^2 \d \nu \right)^k \d \mu - J_k \right| <
\eps.
$$
Now we have
\begin{eqnarray}
&&\int \left( \int \left[ \sum_{n=1}^N \tH \circ F^n (x,y) \right]^2 \d \nu \right)^k \d \mu \nonumber\\
&=& \sum_{n_1, n_2, ..., n_{2k} = 1}^N \sum_{i_1,i_2 ..., i_{2k} \in I}
\sum_{z_1,z_2, ..., z_{2k} \in \integers}  \label{eq:sum2dim1} \\
&&\int 
p(x) 
\left[ \prod_{j=1}^k
\int \tH (f^{n_{2j-1}}(x),G_{\tau_{n_{2j-1}}(x)} (y)) 
\tH (f^{n_{2j}}(x),G_{\tau_{n_{2j}}(x)} (y)) \d\nu \right]\d\mu \nonumber
\end{eqnarray}
where
$$
p(x) = p_{\eta,\underline n ,\underline i, \underline z}(x) = 
\left[ \prod_{j=1}^{2k}
\one_{f^{n_j}x \in X_{i_j}} \right]
\left[ \prod_{j=1}^{2k}
\one_{\tau_{n_j}(x) \in [z_j, z_j + 1]\eta}
\right].
$$
Let us order the numbers $n_1,...,n_{2k}$ as $m_1 \leq m_2 \leq ... \leq m_{2k}$ and denote $m_0 = 0$.
As before, we write the sum in \eqref{eq:sum2dim1} as $T_1 + T_2$, where $T_1$
corresponds to the terms satisfying
\begin{itemize}
\item[(A1')] for every $j=1,...,2k$,
$m_j - m_{j-1} > \eta'N $
\item[(A2')] $|z_{2j-1}| < L \sqrt N /\eta$ for all $j=1,...,k$
\item[(A3')] $|z_{2j-1} - z_{2j}| < L / \eta$ for all $j=1,...,k$
\end{itemize} and write $T_1 = \widetilde \sum (...)$. 
As in \eqref{eq:T_1},
\begin{equation}
\label{eq:T_1'}
T_1 \approx 
\widetilde \sum 
 \int 
p(x) \d\mu \prod_{j=1}^k \int \tH (x_{i_{2j-1}},G_{z_{2j-1} \eta} (y)) 
\tH (x_{i_{2j}},G_{z_{2j} \eta} (y)) \d\nu ,
\end{equation}
where $\approx$ now means that the difference between the LHS and the RHS an be made smaller than
${\eps \Lambda^k N^{3k/2}}/{10}$ provided that $\eta$ and $\eta'$ are small enough and $L$ and $N$ are 
large enough.
To compute $\int p(x) d \mu$, we use the MMLLT for times $m_1 < m_2 < ... < m_{2k}$. 
Note however that the range of summation for different $z_j$'s are different and so it is 
important to keep track of the index $i$ so that $n_i = m_j$. To this end, 
define to permutation $\rho$ (uniquely defined by the tuple $(n_1,...,n_{2k})$
if (A1') holds)
so that $m_j = n_{\rho(j)}$.
Writing $\rho(0) = 0$ and $z_0= 0$ and using the MMLLT we obtain
\begin{eqnarray}
T_1 &\approx &
\frac{1}{\varsigma_1^{2k}} \widetilde \sum \eta^{2k}
\left[
\prod_{j=1}^{2k} \frac{1}{\sqrt{m_{j} - m_{j-1}}} 
\varphi \left( \frac{(z_{\rho(j)} - z_{\rho(j-1)}) \eta}{\varsigma_1 \sqrt{m_{j} - m_{j-1}}}\right) 
\right] \label{eq:sum3dim1} \times\\
&& \left[ \prod_{j=1}^{2k} \mu (X_{i_j})\right]
\left[ \prod_{j=1}^k
\int \tH (x_{i_{2j-1}},G_{z_{{2j-1}}} (y)) 
\tH (x_{i_{2j}},G_{z_{{2j}}} (y))  \d\nu \right] \label{eq:sum4dim1}.
\end{eqnarray}
By (A1') and (A3'), 
$\rho(l)$ can be replaced by $2 \lceil \rho(l) / 2 \rceil -1$ (the biggest odd integer not bigger than
$\rho(j)$) for $l=j-1,j$ in the subscripts of $z$
in \eqref{eq:sum3dim1}. Consequently, $z_{2j}$ only appears 
in \eqref{eq:sum4dim1}. As before, we compute
$$
\sum_{i_1,...,i_{2k}}
\left[ \prod_{j=1}^{2k} \mu (X_{i_j})\right]
\eta^k \sum_{\substack{z_2, z_4, ..., z_{2k}:\\ |z_{2j} - z_{2j -1}| < L/\eta}}
\left[ \prod_{j=1}^k
\int \tH (x_{i_{2j-1}},G_{z_{{2j-1}}} (y)) 
\tH (x_{i_{2j}},G_{z_{{2j}}} (y))  \d\nu \right] $$$$
= \varsigma_2^{2k}+o(1).
$$
Thus we arrive at
$$
T_1 \approx
\frac{\varsigma_2^{2k}}{\varsigma_1^{2k}} \widehat \sum \eta^{k}
\left[
\prod_{j=1}^{2k} \frac{1}{\sqrt{m_{j} - m_{j-1}}} 
\varphi \left( \frac{(z_{2 \lceil \rho(j) / 2 \rceil -1} - z_{2 \lceil \rho(j-1) / 2 \rceil -1}) \eta}{\varsigma_1 \sqrt{m_{j} - m_{j-1}}}\right) 
\right],
$$
where $\widehat \Sigma$ refers to the sum for $n_1,...,n_{2k}$, $z_1, z_3,...,z_{2k-1}$, satisfying (A1'), (A3').
Next, observe that $u: \{1,...,2k\} \to \{1,3...,2k-1\}$
defined by $u(j) = 2 \lceil \rho(j) / 2 \rceil -1$ is a two-to-one mapping. 
Let $U$ be the set of all such mappings. 
The summand in the last displayed formula only depends on $(n_1,...,n_{2k})$
through $(m_1,...,m_{2k})$ and $u$. Furthermore, for any given $(m_1,...,m_{2k})$ and $u$, 
there are exactly $2^k$ corresponding tuples $(n_1,...,n_{2k})$. Thus
$$
T_1 \approx 2^k
\frac{\varsigma_2^{2k}}{\varsigma_1^{2k}} 
\sum_{u \in U} 
\sum_{\substack{m_1,...m_{2k}\\ \text{satisfying (A1')}}}
\sum_{\substack{z_1,..,z_{2k-1} \\ \text{satisfying (A3')}}}
\eta^{k}
\left[
\prod_{j=1}^{2k} \frac{1}{\sqrt{m_{j} - m_{j-1}}} 
\varphi \left( \frac{(z_{u(j)} - z_{u(j-1)}) \eta}{\varsigma_1 \sqrt{m_{j} - m_{j-1}}}\right) 
\right].
$$
Now we are ready to replace the last two sums (Riemann sums) 
by the corresponding Riemann integral with $t_j \sim m_j / N$ and 
$w_l \sim z_{2l-1} \eta / \sqrt N$. To simplify the notation a little, we denote 
$v(j) = (u(j) + 1) /2$, thus $v$ is a 2-to-1 mapping from $\{1,2,...,2k\}$ to $\{1,2,...,k\}$
(the set of all such mapping is denoted by $V$). We obtain
\begin{eqnarray*}
T_1 &\approx& 2^k
\frac{\varsigma_2^{2k}}{\varsigma_1^{2k}} N^{3k/2}
\sum_{v \in V} \\
&&
\int_{|w_1|,...,|w_k| < L}
\int_{\substack{0 =t_0 < t_1 < ... < t_{2k} < 1: \\ t_j - t_{j-1} < \eta'}}
\prod_{j=1}^{2k} \frac{1}{\sqrt{t_j - t_{j-1}}}
\prod_{j=1}^{2k} \varphi
\left( \frac{w_{v(j)} - w_{v(j-1)}}{ \varsigma_1 \sqrt{t_j - t_{j-1}}}
\right) \d t\d w.
\end{eqnarray*}
Substituting $w$ to $w/\varsigma_1$ in the above integrals, we obtain
$$
T_1 \approx
\frac{\varsigma_2^{2k}}{\varsigma_1^{k}}
 \pi^{k /2} N^{3k/2} J_k.
$$
As in Lemma \ref{lem:firstmoment}, $T_2$ is negligible.
This completes the proof of Lemma \ref{lem:highermoment}.
\end{proof}

To finish the proof of Proposition \ref{prop:localtime},
it remains to verify \eqref{eq:analyticmgf}. To this end, first note that since $\varphi$ decays quickly,
there is a constant $M$ so that for every $v \in V$ and every $t_1,...,t_{2k} \in (0,1]$,
$$
\int_{\reals^k} \prod_{j=1}^{2k} \varphi
\left( \frac{w_{v(j)} - w_{v(j-1)}}{ \varsigma_1 \sqrt{t_j - t_{j-1}}}
\right) \d w < M^k .
$$
Noting that $|V| = (2k)! / 2^k$, we have
$$
|J_{k}| \leq C^k (2k)! \int_{0 < t_1 < ... < t_{2k} < 1}
\prod_{j=1}^{2k} \frac{1}{\sqrt{t_j - t_{j-1}}} \d t.
$$
The above integral can be computed explicitly and is equal to 
$\DS \frac{\Gamma(1/2)^{2k}}{\Gamma(k+1)}=
\frac{\pi^k}{k!}.$ 
Noting that $(2k)! / (k!)< C^k k!$, we find that 
$\DS
|J_{k}| \leq C^k k!
$
whence \eqref{eq:analyticmgf} follows. This completes
the proof of Proposition \ref{prop:localtime}.

\section{Polynomially mixing flows in dimension one.}
\label{ScPoly}
In this section, we extend the result of Theorem \ref{ThCLT1} 
to some flows with polynomial mixing rates. 
We use the setting of \cite{D04}. Recall that a flow $G_t$ is called {\em partially hyperbolic} if
there is a $G_t$ invariant splitting $TY=E_u\oplus E_{cs}$ and positive constants 
$C_1, C_2, \lambda_1, \lambda_2$ such that for each $t\geq 0$

(i) $\|dG_{-t}|E_u\|\leq C_1 e^{-\lambda_1 t};$

(ii) For any unit vectors $v_u\in E_u,$  $v_{cs}\in E_{cs}$ we have
$\|dG_{t}|v_{cs} \|\leq C_2 e^{-\lambda_2 t} \|dG_t v_u\|.$ 

For partially hyperbolic flows the leaves of $E_u$ are tangent to the leaves of an absolutely 
continuous foliation $W^u.$ 

Fix constants $R, \brv, \brC_1, \alpha_1$. Let $\fA(R, \brv, \brC_1, \alpha_1)$ 
denote the collection of 
sets $D$ which belong to one leaf of $W^u$ and satisfy 
$$\diam(D)\leq R, \quad \mes(D)\geq \brv, \quad 
\mes(\partial_\eps D)\leq \brC_1 \eps^{\alpha_1} $$
for all $\eps > 0 $,
where $\partial_\eps D$ is the $\eps$ neighborhood of the boundary of $D.$ 
We say that the sets from $\fA(R, \brv, \brC_1, \alpha_1)$ have {\em bounded geometry}.

Fix $\brC_2, \alpha_2>0.$ Let $\fM(R, \brv, \brC_1, \alpha_1, \brC_2, \alpha_2)$ 
denote the set of linear functionals of the form 
$$\EXP_{\ell_{D, \rho}}(A)=\int_D A(y) \rho(y) \d y, $$
where $D\in \fA(R,  \brv, \brC_1, \alpha_1),$ $\rho$ is a probability density on $D$,
$\ln \rho\in C^{\alpha_2}(D),$ 
$\|\ln \rho\|_{C^r} \leq \brC_2.$
We shall often use the existence of {\em almost Markov decomposition} established in
\cite[Section 2]{D04}: if $R$, $\brC_1,$ and $\brC_2$ are large enough and
$\alpha_1, \alpha_2$ and $\brv$ are small enough, then
given $\ell\in \fM(R, \brv, \brC_1, \alpha_1, \brC_2, \alpha_2)$ and $t>0$ we can decompose
\begin{equation}
\label{AMFlows}
 \EXP_\ell(A\circ G_t)=\sum_s c_s \EXP_{\ell_s} (A)+O\left(\|A\|_{C^0} \theta^t\right) 
\end{equation} 
with $\ell_s\in\fM(R, \brv, \brC_1, \alpha_1, \brC_2, \alpha_2)$ and 
$\DS \sum_s c_s=1+O(\theta^t)$ for some $\theta<1.$

We say that a measure is {\em u-Gibbs} if its conditional measures on unstable leaves 
have smooth densities. The existence of almost Markov decomposition implies that u-Gibbs measures belong
to the convex hull of $\fM(R, \brv, \brC_1, \alpha_1, \brC_2, \alpha_2)$  for appropriate choice of
$R, \brv, \brC_1, \alpha_1, \brC_2, \alpha_2$ (see \cite[\S 11.2]{BDV}).

From now on we shall fix constants involved in the definition of $\fM(R, \brv, \brC_1, \alpha_1, \brC_2, \alpha_2)$  which satisfies \eqref{AMFlows} and whose convex hull contains u-Gibbs
measures.
For the sake of simplicity, we will write $\fM$ instead of $\fM(R, \brv, \brC_1, \alpha_1, \brC_2, \alpha_2).$

\medskip

We say that the {\em unstable leaves become equidistributed
at rate $a(t)$} if there is $r>0$ such that for all $\ell\in \fM$
we have
\begin{equation}
\label{EquiWU}
 \left|\EXP_\ell(A\circ G_t)-\nu(A)\right|\leq a(t) \|A\|_{C^r}
\end{equation}
where $\nu$ is the reference invariant measure for $G_t.$
We note that if \eqref{EquiWU} holds with $a(t)\to 0$, then $\nu$ belongs to the convex hull 
of $\fM$, whence it is a u-Gibbs measure (in fact, it is the unique SRB measure for $G_t$,
see \cite[Corollary 2]{D04}).

\begin{theorem}
\label{ThCLT1Poly}
Theorem \ref{ThCLT1} remains valid if the assumption that $G$ enjoys exponential mixing of all 
orders is replaced by the assumption that $G$ is partially hyperbolic 
and unstable leaves become equidistributed
with rate $K t^{-\gamma}$ for some $\gamma>1$ and $K > 0$.
\end{theorem}

\begin{corollary}
Suppose that $\tau$ satisfies the assumption of Theorem \ref{ThCLT1}, 
$G_t$ is a topologically  transitive Anosov flow whose stable and unstable foliations are not jointly integrable,
and $\nu$ is the SRB measure. Then \eqref{KeSp} holds for sufficiently smooth functions.
 \end{corollary}

Indeed, according to \cite[Theorem 3]{D98}, for Anosov flows, unstable leaves are 
equidistributed on $C^r$  at rate $K t^{-\gamma(r)}$ where $\gamma(r)\to\infty$ as $r\to\infty.$

The proof of Theorem \ref{ThCLT1Poly} requires a modification of Proposition \ref{PrBG}.
We shall use the same notation as in that proposition. 

\begin{proposition}
\label{PrLFPoly}
Suppose that $G_t$ is partially hyperbolic 
with unstable leaves  equidistributed
at rate $K t^{-\gamma}$ for some $\gamma>1.$
Suppose that for some $0<\delta<1/20$ such that 
\begin{equation}
\label{GammaDelta}
\gamma-1>\frac{16\delta}{1-20\delta} 
\end{equation}
and $\kappa>0$ we have\medskip

(a)  $\fm_N( t \in \reals: |t| \leq N^{(1/2)+\delta}) \leq N^{1/4 + \delta}$ and 
 $\fm_N( t: |t| > N^{(1/2)+\delta}) =0$;

(b) for every $t \in \reals$, $\fm_N([t, t+1])\leq N^{\delta-(1/4)};$

(c) $\DS \iint_{|t_1-t_2|>\sqrt{\ln N}} \frac{\d\fm_{N}(t_1)\d\fm_{N}(t_2)}{|t_1-t_2|^\gamma}
\leq \ln^{-\kappa} N;$

(d) $\DS \lim_{N\to\infty} \nu(\cS_N^2)=1.$

Then $\cS_N\Rightarrow \cN(0,1)$ as $N\to \infty.$ 
\end{proposition}

\begin{lemma}
\label{LmRandMes1ph}
Under the assumptions of Theorems \ref{ThCLT1Poly},
there are subsets $X_N\subset X$ such that
$\DS \lim_{N\to\infty} \mu(X_N)=1$ and for any sequence $x_N\in X_N$
the measures $\{\fm_N(x_N)\}$, defined by \eqref{DefMND1}, 
satisfy the conditions of Proposition \ref{PrLFPoly}.
\end{lemma}

\begin{proof}
The fact that  conditions (a) and (b) hold for arbitrary $\delta>0$ are verified as before. Condition 
(d) is immediate from the definition of $\fm_N.$ In order to verify condition (c) we note that by 
Proposition \ref{prop:localtime}, for any $\kappa >0$, $\mu(x: V_N(x)<N^{3/2} \ln^{-\kappa} N)\to 0$ as $N\to\infty.$
Therefore it is enough to check that if $\delta$ and $\kappa$ are sufficiently small, then
$$
\mu(x: \cR_N(x) <N^{3/2} \ln^{-2\kappa} N)
$$
is close to 1,
where
$$ \cR_N (x):=\sum_{n_1, n_2\in \integers: |n_1-n_2|>\sqrt{\ln N}} \frac{\ell(x, n_1, N)  \ell(x, n_2, N)}
{|n_1-n_2|^\gamma} $$
and $\ell(x,t, N)$ is the local time defined by \eqref{DefLT}.
Note that 
$$ \mu\left(\ell(n_1, N)  \ell(n_2, N)\right)= 
\sum_{1 \leq j_1, j_2 \leq N} 
\mu(x: |\tau_{j_1}(x) - n_1| \leq 1, | \tau_{j_2}(x) - n_2 | \leq 1). $$


Therefore, using the anticoncentration large deviation bound, we conclude that
$$\mu(\cR_N)\leq C
\sum_{1 \leq j_1\leq j_2 \leq N} \sum_{|n_1-n_2|\geq \sqrt{\ln N}} \frac{1}{\sqrt{j_1} \sqrt{j_2-j_1+1} }
\Theta\left(\frac{n_1}{\sqrt{j_1}}\right) 
\frac{1}{|n_2-n_1|^\gamma} $$
$$
\leq C \frac{\sqrt{N}}{\ln^{(\gamma -1)/2} N } 
\sum_{j_1} \sum_{n_1} 
\frac{1}{\sqrt{ j_1}}\;
\Theta\left(\frac{n_1}{\sqrt{j_1}}\right) 
\leq C \frac{\sqrt{N}}{\ln^{ (\gamma -1)/2} N } 
\sum_{j_1} 
 \int_0^{\infty} \Theta\left(r \right) \d r 
\leq C \frac{N^{3/2}}{\ln^{(\gamma -1)/2} N } .
$$
Now the result follows by the Markov inequality provided that $ \gamma -1>4\kappa.$ 
\end{proof}

Thus in order to prove Theorem \ref{ThCLT1Poly} it suffices to establish Proposition \ref{PrLFPoly}.

\begin{proof}[Proof of Proposition \ref{PrLFPoly}.]
We divide the interval $[-N^{1/2+\delta}, N^{1/2+\delta}]$ into big blocks of size $N^{\beta_1}>0$ 
separated by small blocks of size $N^{\beta_2}$ where the parameters 
$\beta_2<\beta_1$ will be chosen later.
Let $J_j$ denote the $j$-th big block and $L$ be the union of the small blocks.
Let $\cS_N'=\int_{t\in L} A_t(G_t y) \d\fm_N(y).$ Note that $\nu(\cS'_N)=0.$ We claim 
that $\nu((\cS_N')^2)\to 0$ as $N\to\infty$ provided that
\begin{equation}
\label{Beta-1}
\beta_1>\beta_2+3\delta. 
\end{equation}
Indeed
\begin{equation}
\label{VarSmallBlocks}
 \nu\left(\left(\cS_N'\right)^2\right)
\leq \iint_{L\times L} \nu(A_{t_1} (A_{t_2}\circ G_{t_2-t_1})) \d\fm_N(t_1) \d\fm_N (t_2).
\end{equation}
According to \cite[Theorem 2]{D04}, 
$\DS \nu(A_{t_1} (A_{t_2}\circ G_{t_2-t_1}))=O\left(|t_2-t_1|^{-\gamma}\right).$ Therefore property (b) 
shows that the integral of the RHS of \eqref{VarSmallBlocks}
with respect to $t_2$ is $O\left(N^{-(1/4)+\delta} \right)$. To integrate with respect to $t_1$ we divide 
$L$ into unit intervals. Noting that the mass of each interval is $O\left(N^{\delta-(1/4)}\right)$ and the number of intervals 
is $O\left(N^{(1/2)+\delta-\beta_1+\beta_2}\right)$, we conclude that
$\DS \nu\left(\left(\cS_N'\right)^2\right)=O\left(N^{\beta_2+3\delta-\beta_1}\right)$ which tends to zero 
under \eqref{Beta-1}.

Thus the main contribution to the variance comes from the big blocks. Let
$$
\tau = N^{1/2 + \delta}, \quad
T_j=
T_j(y) =\int_{J_j} 
A_t(G_{\tau + t} y) \d\fm_N(t), \quad
S_{N,j}=\sum_{k=1}^j T_k.$$

Fix any $\xi \in \reals$. We will 
show that there is a sequence $\eps_N \to 0$, depending only 
$|\xi|$ so that for each 
$\ell\in \fM$
we have 
$$ |\EXP_\ell\left(e^{i\xi S_{N, \brN}}\right) - e^{-\xi^2/2}| \leq \eps_N $$
where $\brN$ is the number of big blocks. Let 
$\DS \Phi_j(\xi)=\EXP_\ell\left(e^{i\xi S_{N, j}}\right)$
with $\Phi_0(\xi) = 1$.

\begin{lemma} 
\label{LmOneBlock}
If 
\begin{equation}
\label{Beta-3}
\beta_1<1/4-2\delta,
\end{equation}
 then
$$ \Phi_j(\xi)=\Phi_{j-1}(\xi) \left[1-\frac{\xi^2}{2} v_j\right]+O\left(v_j N^{-\delta}+w_j
 + 
\heps_N\right), $$
where
$$ v_j\!=\!\nu(T_j^2), \quad 
w_j\!=\!\iint_{(t_1, t_2)\in J_j \atop |t_1-t_2|\geq \sqrt{\ln N}} 
\frac{\d\fm_N (t_1) \d\fm_N (t_2)}{|t_1 - t_2|^\gamma}, 
\quad \heps_N\!=\!N^{\delta-1/4-\beta_2(\gamma-1)}.$$ 
\end{lemma}
Note that by property (b) and \eqref{Beta-3}, we have
$T_j\!\!=\!\!O(N^{-\delta}).$ 
Hence $v_j\!\!=\!\!O(N^{-2\delta})$ and so the estimate of 
Lemma~\ref{LmOneBlock} can be rewritten as
$$ \Phi_j(\xi)=\Phi_{j-1}(\xi) e^{-\xi^2 v_j/2} +O\left(v_j N^{-\delta}+w_j+
\heps_N\right) .$$
Repeating this process, we obtain
\begin{equation}
\label{CharEll}
 \EXP_\ell\left(e^{i \xi \cS_{N, \brN}}\right)=\Phi_\brN(\xi)=
\exp\left[-\frac{\xi^2}{2} \sum_{j=1}^N v_j \right]+
O\left(\sum_j \left[v_j N^{-\delta}+w_j\right]+\brN \heps_N\right). 
\end{equation}
Next we claim that
\begin{equation}
\label{VarSumU}
\sum_{j=1}^\brN v_j=1+o_{N\to\infty}(1)
\end{equation}
provided that
 \begin{equation}
 \label{Beta-2}
\beta_2(\gamma-1)>2\delta.
\end{equation}

Indeed, 
$\DS 1+o(1)\!\!=\!\!\nu(S_{N, \brN}^2)\!\!=\!\!\sum_{j=1}^\brN v_j\!+\!
\sum_{j_1\neq j_2} \nu\left(T_{j_1} T_{j_2}\right). $
The second term here is at most
$$ C \sum_{n_1, n_2: |n_1-n_2|\geq N^{\beta_2} }
\frac{\fm_N([n_1, n_1+1]) \fm_N([n_2, n_2+1])}{|n_1-n_2|^\gamma} .$$
Summing over $n_2$ using assumption (b) we are left with
$$ C \sum_{n_1} \fm_N([n_1, n_1+1]) N^{\delta-(1/4)-\beta_2(\gamma-1)}$$
$$=C\fm_N([-N^{1/2+\delta}, N^{1/2+\delta}]) N^{\delta-(1/4)-\beta_2(\gamma-1)}
\leq \brC N^{2\delta-\beta_2(\gamma-1)}=o_{N\to \infty} (1),
 $$
where in the last inequality we also used assumption (a).

\noindent 
 This proves that \eqref{Beta-2} implies \eqref{VarSumU}.
  \eqref{VarSumU} shows in particular that \\$\DS \sum_j v_j N^{-\delta}\!\!=\!\!O(N^{-\delta}). $
 Also $\DS \sum_j w_j\!=\!o(1)$ due to assumption (c) of Proposition \ref{PrLFPoly}, while
 $$\brN \heps\!\leq\!N^{1/4+2\delta-\beta_1-\beta_2(\gamma-1)}\!=\!o(1)$$
 provided that
\begin{equation}
\label{Beta-4} 
\beta_1+\beta_2(\gamma-1)> \frac14 +2\delta.
\end{equation}

 Plugging these estimates into \eqref{CharEll} we
 conclude that for all $\ell \in \fM$ we have
 \begin{equation}
 \label{CharFinal}
  \EXP_\ell\left(e^{i\xi S_{N, \brN}}\right)=e^{-\xi^2/2}+o_{N\to\infty}(1)
  \end{equation}
 if $\beta_1$ and $\beta_2$ satisfy \eqref{Beta-1}, \eqref{Beta-3}, \eqref{Beta-2}, and
 \eqref{Beta-4}. Thus we need $\beta_1$ and $\beta_2$ to satisfy
 $$ \beta_1<\frac{1}{4}-2\delta, \quad \beta_2<\beta_1-3\delta, \quad
 \beta_1+\beta_2(\gamma-1)> \frac{1}{4}+2\delta,
 \quad \beta_2(\gamma-1)>2\delta. $$
Since $\beta_1$ can be chosen arbitrary close to $\frac{1}{4}-2\delta$ and $\beta_2$
can be chosen arbitrarily close to $\beta_1-3\delta$, the above inequalities are compatible
if \eqref{GammaDelta} holds.
 It then follows that  \eqref{CharFinal} holds on the convex hull of $\fM$ 
 which includes $\nu.$ This completes the proof of Proposition \ref{PrLFPoly} modulo Lemma \ref{LmOneBlock}.
\end{proof}

\begin{proof}[Proof of Lemma \ref{LmOneBlock}]
Let $J_j=[n_j^-, n_j^+].$ Denote $m_j=\tau+\frac{n_{j-1}^++n_{j}^-}{2}.$
We use the almost Markov decomposition \eqref{AMFlows}:
$$ \EXP_\ell(A\circ G_{m_j})=\sum_s c_s \EXP_{\ell_s} (A)+O(\teps_N) $$
where $\teps_N=\theta^{N^{\beta_2}}$, $\ell_s=(D_s, \rho_s)\in \fM$ and $\DS \sum_s c_s=1-O(\teps_N).$

Fix arbitrary $y_s\in G_{ -m_j} D_s.$ Then 
$$ \EXP_\ell\left(e^{i\xi S_{N, j}}\right)=\sum_s c_s \EXP_{\ell_s} \left(e^{i\xi [S_{N, j-1} (y_s)+\tT_j(y)]} \right)+
O(\teps_N)
$$
where 
$$\tT_j(y)=T_j(G_{-m_j}y) =\int_{J_j} A_t(G_{t-m_j} y) \d\fm_N(t). $$
Next
$$ \EXP_{\ell_s} \left(e^{i\xi \tT_j}\right)=
1+i\xi \EXP(\tT_j)-\frac{\xi^2}{2}\EXP_{\ell_s} (\tT_j^2)+O\left(\EXP_{\ell_s}\left(\left|\tT_j\right|^3\right)\right).$$
Note that $\DS \EXP(\tT_j)=O\left(N^{\beta_2(\gamma-1)-(1/4)+\delta}\right)$
due to the equidistribution of unstable leaves.

To estimate $\EXP_{\ell_s} (\tT_j^2)$ we split 
$$ \EXP_{\ell_s} (\tT_j^2)=\iint_{J_j\times J_j} 
\EXP_{\ell_s} [A_{t_1} (G_{t_1-m_j} y) A_{t_2} (G_{t_2-m_j} y)  ]
\d \fm_N(t_1) \d \fm_N(t_2)=I+\RmII$$
where $I$ includes the terms where $|t_1-t_2|\leq \sqrt{\ln N}$ and $\RmII$ includes the other
terms.

According to \cite[Theorem 2]{D04}, 
$$\EXP_{\ell_s} \left((A_{t_1} \circ G_{t_1-m_j})(A_{t_2} \circ G_{t_2-m_j})\right)=
O\left(|t_2-t_1|^{-\gamma}\right). $$
Therefore
$\RmII=O(w_j).$

To estimate $I$, we note that 
$$\DS \EXP_{\ell_s} \left((A_{t_1} \circ G_{t_1-m_j})(A_{t_2} \circ G_{t_2-m_j})\right)=
\EXP_{\ell_s} \left((D_{t_1, t_2} \circ G_{t_1-m_j})\right) 
$$ where
$D_{t_1, t_2}=A_{t_1} (A_{t_2} \circ G_{t_2-t_1})$ with 
$\DS \|D_{t_1, t_2}\|_{C^r} \leq K^{t_2-t_1}\leq N^{\delta/2}.$

Hence using the equidistribution of  unstable leaves, we obtain
$$ \EXP_{\ell_s} \left((A_{t_1} \circ G_{t_1-m_j})(A_{t_2} \circ G_{t_2-m_j})\right)
=\nu(D_{t_1, t_2})+O\left(N^{\delta/2} (t_1-m_j)^{-\gamma}\right)$$$$
=\nu(A_{t_1} (A_{t_2} \circ G_{t_2-t_1}))+
O\left(N^{\delta/2}(t_1-m_j)^{-\gamma}\right).$$
It follows that
$$I=\iint_{t_1, t_2\in J_j\atop |t_1-t_2|\leq \sqrt{\ln N}} \nu(A_{t_1} (A_{t_2} G_{t_2-t_1}))
\d\fm_N(t_1) \d\fm_N(t_2)
+O\left(N^{(5\delta/2)-(1/2)-\beta_2 (\gamma-1)} \sqrt{\ln N}\right)
$$$$
=v_j+O(w_j)+O\left(N^{3\delta-(1/2)-\beta_2(\gamma-1)} \right). $$

Finally 
$$ \EXP_{\ell_s} (|\tT_j|^3)\leq N^{-\delta} \EXP_{\ell_s} (|\tT_j|^2)=
O\left(\left[v_j+w_j\right] N^{-\delta} +N^{2\delta-(1/2)-\beta_2 (\gamma-1)} \right).$$
Summing over $s$ and using the fact that
$\DS \sum_s c_s \EXP_{\ell_s}\left(e^{i\xi S_{N,{j-1}}(y_s)}\right)=\EXP_\ell\left(e^{i\xi S_{N, j-1}}\right)+O(\teps_N) $,
we obtain the result.
\end{proof}

\section*{Acknowledgement} 
C.D. was partially supported by Nankai Zhide Foundation, and AMS-Simons travel grant. 
D.D. was partially supported by NSF DMS 1956049. 
A.K. was partially supported by NSF DMS 1956310.
P.N. was partially supported by NSF DMS 2154725.

\end{document}